\documentclass[11pt]{amsart}
\usepackage[margin=1.35in]{geometry}
\usepackage{amscd,amsmath,amsxtra,amsthm,amssymb,stmaryrd,xr,mathrsfs,mathtools,enumerate,commath, comment}
\usepackage{stmaryrd}
\usepackage{xcolor}
\usepackage{commath}
\usepackage{comment}
\usepackage{tikz-cd}
\usepackage{longtable} 
\usepackage{pdflscape} 
\usepackage{booktabs}
\usepackage{hyperref}
\definecolor{vegasgold}{rgb}{0.77, 0.7, 0.35}
\definecolor{darkgoldenrod}{rgb}{0.72, 0.53, 0.04}
\definecolor{gold(metallic)}{rgb}{0.83, 0.69, 0.22}
\hypersetup{
 colorlinks=true,
 linkcolor=darkgoldenrod,
 filecolor=brown,      
 urlcolor=gold(metallic),
 citecolor=darkgoldenrod,
 }
\newtheorem{lthm}{Theorem}

\newtheorem{lcor}{Corollary}

\usepackage[all,cmtip]{xy}

\DeclareFontFamily{U}{wncy}{}
\DeclareFontShape{U}{wncy}{m}{n}{<->wncyr10}{}
\DeclareSymbolFont{mcy}{U}{wncy}{m}{n}
\DeclareMathSymbol{\Sh}{\mathord}{mcy}{"58}
\usepackage[T2A,T1]{fontenc}
\usepackage[OT2,T1]{fontenc}

\newtheorem{theorem}{Theorem}[section]
\newtheorem{lemma}[theorem]{Lemma}

\newtheorem*{theorem*}{Theorem}
\newtheorem*{ass*}{Assumption}
\newtheorem{definition}[theorem]{Definition}

\newtheorem{conjecture}[theorem]{Conjecture}
\newtheorem{proposition}[theorem]{Proposition}

\newcommand{\Z}{\mathbb{Z}}

\newcommand{\Q}{\mathbb{Q}}

\newcommand{\cO}{\mathcal{O}}

\newcommand{\mb}[1]{\mathbf{#1}}

\newcommand{\op}[1]{\operatorname{#1}}

\numberwithin{equation}{section}

\begin{document}

\title[Number of number fields with prescribed Galois group]{Upper bounds for the number of number fields with prescribed Galois group}
\author[H.~Mishra]{Hrishabh Mishra}
\address[Mishra]{Chennai Mathematical Institute, H1, SIPCOT IT Park, Kelambakkam, Siruseri, Tamil Nadu 603103, India}
\email{hrishabh@cmi.ac.in}

\author[A.~Ray]{Anwesh Ray}
\address[Ray]{Chennai Mathematical Institute, H1, SIPCOT IT Park, Kelambakkam, Siruseri, Tamil Nadu 603103, India}
\email{anwesh@cmi.ac.in}


\maketitle

\begin{abstract}
Let $n$ be a positive integer and $G$ be a transitive permutation subgroup of $S_n$. Given a number field $K$ with $[K:\Q]=n$, we let $\widetilde{K}$ be its Galois closure over $\Q$ and refer to $\op{Gal}(\widetilde{K}/\Q)$ as its Galois group. We  may identify this Galois group with a transitive subgroup of $S_n$. Given a real number $X>0$, we set $N_{n}(X;G)$ to be the number of such number fields $K$ for which the absolute discriminant is bounded above by $X$, and for which $\op{Gal}(\widetilde{K}/\Q)$ is isomorphic to $G$ as a permutation subgroup of $S_n$. We prove an asymptotic upper bound for $N_n(X;G)$ as $X\rightarrow\infty$. This result is conditional and based upon the non-vanishing of certain polynomial determinants in $n$-variables. We expect that these determinants are non-vanishing for many groups, and demonstrate through some examples how they may be computed. 
\end{abstract}

\section{Introduction}
\subsection{Background and historical remarks}
\par Let $n$ be a positive integer and $X$ be a positive real number. For a number field $K$, we set $\Delta_K$ to denote the discriminant of $K$ over $\Q$. Set $N_n(X)$ to be the number of number fields $K$ with $[K:\Q]=n$ and with $|\Delta_K|\leq X$. The Hermite Minkowski theorem implies that $N_n(X)$ is finite. It is expected that $N_n(X)\sim c_n X$, where $c_n$ is a constant that depends on $n$ (cf. \cite[p.723]{ellenberg2006number}). This conjecture has only been established for $n\leq 5$ (cf. \cite{davenport1971density, bhargava2005density,bhargava2008density, bhargava2010density}). However, not much is known for $n\geq 6$. From a different perspective, there has been significant interest in obtaining asymptotic upper bounds for $N_n(X)$ (as $X\rightarrow \infty$).
\begin{itemize}
    \item Schmidt \cite{schmidt1995number} showed that $N_n(X)\ll X^{\frac{n+2}{4}}$. 
    \item Ellenberg and Venkatesh \cite{ellenberg2006number} obtained an exponent that is $O\left(\op{exp}\left(c\sqrt{\log n }\right)\right)$, where $c>0$ is a constant.
    \item Couveignes \cite{couveignes2020enumerating} shows that $N_n(X)\ll X^{c(\log n)^3}$, for an undetermined constant $c>0$.
    \item Finally, Lemke-Oliver and Thorne \cite{lemke2022upper} show that $N_n(X)\ll X^{c(\log n)^2}$. In this case, the constant $c$ can be taken to be $1.564$.
\end{itemize}
It is natural to consider asymptotics in the setting in which the Galois group of the Galois closure of the number field in question is prescribed. Given a number field $K$ with $[K:\Q]=n$, set $\widetilde{K}$ to denote the Galois closure of $K$. We shall by abuse of notation simply refer to $\op{Gal}(\widetilde{K}/\Q)$ as the \emph{Galois group of $K$}. Let $\sigma_1, \dots, \sigma_n$ be an enumeration of the embeddings of $K$ into $\bar{\Q}$, with the convention that $\sigma_1=\op{Id}$. Let $[n]$ be the set of integers in the range $[1,n]$ and let $S_n$ be the set of bijections $[n]\rightarrow [n]$. We shall identify $\sigma_i$ with $i\in [n]$ and the set of permutations of $\sigma_1, \dots, \sigma_n$ with $S_n$. Given $\sigma\in \op{Gal}(\widetilde{K}/\Q)$ and $i\in [1, n]$, the composite $\sigma \sigma_i$ is well defined  and is also an embedding of $K$ into $\bar{\Q}$. In this way, the Galois group $\op{Gal}(\widetilde{K}/\Q)$ permutes the embeddings $\sigma_1, \dots, \sigma_n$. This gives rise to a permutation representation 
\[\Phi: \op{Gal}(\widetilde{K}/\Q)\hookrightarrow S_n,\] with \[\Phi(\sigma)(\sigma_i):=\sigma \sigma_i.\] It is via this representation that one may view $\op{Gal}(\widetilde{K}/\Q)$ as a transitive subgroup of $S_n$.
\par We fix a transitive subgroup $G$ of $S_n$. Let $X$ be a positve real number, and set  
\[N_{n} (X;G):=\#\{K\mid [K:\Q]=n, |\Delta_K|\leq X, \op{Gal}(\widetilde{K}/\Q)\simeq G\}, \]
where the isomorphism $\op{Gal}(\widetilde{K}/\Q)\simeq G$ is that of permutation subgroups of $S_n$. We note that the quantity $N_n(X;G)$ depends not only on the group $G$ but also on the embedding of $G$ into $S_n$. However, this is suppressed in our notation. Given a conjugacy class $C$ of $G$, let $\op{ind}(C)$ denote $\op{ind}(g)$, where $g\in C$. For any group $G\neq 1$, set $G^\#:=G\backslash\{1\}$, and set
\[a(G):=\left(\op{min}\{\op{ind}(g)\mid g\in G^\#\}\right)^{-1}.\] Malle made predictions about the asymptotic growth of $N_n(X;G)$ as a function of $X$. We state the \emph{weak version} of his conjecture below, cf. \cite[p.316]{malle2002distribution} for further details.
\begin{conjecture}[Malle's conjecture -- weak form]
    Let $G\subseteq S_n$ be a transitive permutation group. Then, for all $\epsilon>0$, there exist constants $c_1(G), c_2(G; \epsilon)>0$ such that 
    \[c_1(G)X^{a(G)}\leq N_{n}(X;G)<c_2(G;\epsilon)X^{a(G)+\epsilon},\]
    for all large enough values of $X$.
\end{conjecture}

We note that when $n=|G|$ and $G\subset S_n$ via the regular representation, $N_n(X;G)$ simply counts the number of Galois extensions $K/\Q$ with $|\Delta_K|\leq X$ and with $\op{Gal}(K/\Q)\simeq G$ (as a permutation subgroup of $S_n$). Consider the special case when $G$ is a finite group with $n:=|G|>4$, and $G\hookrightarrow  S_{n}$ the regular representation. Then, Ellenberg and Venkatesh \cite[Proposition 2.8]{ellenberg2006number} showed that for any $\epsilon>0$, one has the asymptotic upper bound $N_n(X;G)\ll X^{3/8+\epsilon}$. Here, the implied constant is allowed to depend on $n$ and $\epsilon>0$. However, when permutation representation $G\hookrightarrow S_n$ is \emph{not} the regular representation, not much is known in general. For the alternating subgroup $A_n\hookrightarrow S_n$, asymptotic upper bounds due to Larson and Rolen \cite{larson2013upper} for $N_n(X;A_n)$ improve upon Schmidt's upper bounds by a factor of $X^{1/4}$.
\subsection{Main results}
\par In order to state our main results precisely, we introduce some further notation. Let $G$ be a transitive subgroup of $S_n$. Our goal is to establish asymptotic upper bounds for $N_{n} (X;G)$, provided some additional conditions are satisfied. We assume without loss of generality that there exists a number field $K$ with $[K:\Q]=n$, such that $\op{Gal}(\widetilde{K}/\Q)\simeq G$.
The subgroup of $S_n$ that fixes the embedding $\sigma_1$ is denoted by $\op{Stab}(\sigma_1)$, and is identified with $S_{n-1}$. With this convention in place, we find that \[\op{Gal}(\widetilde{K}/K)=\op{Gal}(\widetilde{K}/\Q)\cap S_{n-1}.\]
Let $H$ be the intersection $G\cap S_{n-1}$. Since $H$ is identified with $\op{Gal}(\widetilde{K}/K)$, we find that \[\widetilde{K}^H=K,\text{ and }[G:H]=[K:\Q]=n.\] Let $N$ denote the normalizer of $H$ in $G$ and let $K_0:=K^N$ be the fixed field of $N$. In particular, $K/K_0$ is a Galois extension. Setting $r:=|N/H|=[K:K_0]$, it shall be assumed without loss of generality that
\[\op{Gal}(K/K_0)=\{\sigma_1, \dots, \sigma_r\}.\] Thus for $j\in [1, r]$, the image of $\sigma_j$ is contained in $K$.

\begin{definition}\label{def of pi}For $i\in [1, n]$ and $j\in [1,r]$, we consider the composite embedding \[ K\xrightarrow{\sigma_j} K\xrightarrow{\sigma_i}\widetilde{K}.\]
Consequently, there is a permutation $\pi_j\in S_n$ such that $\sigma_i \sigma_j=\sigma_{\pi_j(i)}$.
\end{definition}
Since $\sigma_1$ is the identity on $K$, it follows that $\pi_1$ is the identity in $S_n$. Also note that $\pi_j(i)\neq i$ unless $j=1$.

\par One may describe the permutations $\pi_j$ intrinsically. Write \[G/H=\bigcup_{i=1}^n g_i H\text{ and } N/H=\bigcup_{j=1}^r g_j H.\] Note that for $i\in [1,n]$ and $j\in [1, r]$, the element $g_j$ normalizes $H$, and hence $g_j H =Hg_j$. Thus we find that for $j\in [1,r]$, the following relationship holds
\[g_i H g_j H=g_ig_j H. \] There is a unique index $k\in [1, n]$ such that 
\[g_i g_j H=g_k H,\] and we set $\pi_j(i):=k$. For a number field $K$ as above, the embeddings $\sigma_i$ coincide with cosets $g_i H$.

\begin{definition}\label{def of tr}
    Given a tuple $\mb{a}=(a_1, \dots, a_r)\in \Z_{\geq 0}^r$, we set
\[\op{Tr}_{\mb{a}}(x_1, \dots, x_n):=\sum_{i=1}^n \left(\prod_{j=1}^r x_{\pi_j(i)}^{a_j}\right)\in \Z[x_1, \dots, x_n].\] We shall refer to such functions as trace functions.
\end{definition}

Assume that $r\geq 3$. Fix integers $k$ and $t$ such that $k\in [2, r)$ and $t\geq 2$. Let $\mathcal{S}_k(r)$ be the  all $k$-element subsets of $[r]$ that contain $1$. Given $\mathcal{B}\in \mathcal{S}_k(r)$, let $\mb{a}(\mathcal{B})=(a_1, \dots, a_r)$ be the vector which is defined so that 
\[a_i:=\begin{cases}
    t\text{ if }i=1;\\
    1\text{ if }i\neq 1\text{ and }i \in \mathcal{B};\\
    0\text{ otherwise. }
\end{cases}\]Set $l:=\# \mathcal{S}_k(r)=\binom{r-1}{k-1}$ and $\mathcal{B}_1, \dots, \mathcal{B}_l$ to be an enumeration of the subsets in $\mathcal{S}_k(r)$. For $i\in [1, l]$, we set $\mb{a}_i:=\mb{a}(\mathcal{B}_i)$. We state our main results below.

\begin{lthm}\label{main thm}
    Let $G$ be a transitive permutation group contained in $S_n$ and set $H:=G\cap S_{n-1}$. We let $N$ be the normalizer of $H$ and set $r:=|N/H|$. For $j\in [1, r]$, we let $\pi_j\in S_n$ be the associated permutation as prescribed by Definition \ref{def of pi}. We make the following assumptions. 
    \begin{enumerate}
        \item\label{condition 1} There exists $k\in [1,r-1]$ such that $l:=\binom{r-1}{k}\geq n$.
        \item\label{condition 2} Let $\mb{a}_1, \dots, \mb{a}_l$ be the integral vectors as in Definition \ref{def of ai}. Then, assume that for a subset $\{\mb{a}_{i_1}, \dots, \mb{a}_{i_n}\}$ of $\{\mb{a}_1, \dots, \mb{a}_l\}$, the Jacobian matrix \[\mathbb{D}=\mathbb{D}(f_1, \dots, f_n):=\left(\frac{\partial f_i}{\partial x_j}\right)_{1\leq i, j\leq n}\] of the trace functions $f_j:=\op{Tr}_{\mb{a}_j}$ has determinant which is not identically $0$.
    \end{enumerate}
    Then, we have that $N_{n}(X;G)\ll X^{k+t-1}$.
\end{lthm}
Let us further specialize Theorem \ref{main thm} to a family of groups $G$ for which condition \eqref{condition 1} is satisfied. We let $m\in \Z_{\geq 1}$ be an integer and $A$ be a transitive permutation subgroup of $S_m$. The convention here is that if $m=1$, then $A$ is trivial. Let $B$ be any finite group and set $q$ to denote its cardinality. Via the regular representation, $B$ is a subgroup of $S_q$. We write $\iota_A:A\hookrightarrow S_m$ and $\iota_B: B\hookrightarrow S_q$ denote the embeddings of $A$ and $B$ into $S_m$ and $S_q$ respectively. Then, $\iota:=\iota_A\times \iota_B$ realizes $G:=A\times B$ as a subgroup of $S_{mq}$. Let $H_1\subseteq S_{m-1}$ be the intersection $A\cap S_{m-1}$, and set $H:=G\cap S_{mq-1}$. It is easy to see that $H=H_1\times 1$. Since $\iota_B$ is the regular representation, $N:=N_G(H)\supseteq N_1\times B$, where $N_1=N_A(H_1)$. Therefore, $r=|N/H|\geq |B|$, and thus the condition \eqref{condition 1} of Theorem \ref{main thm} is satisfied if \[\binom{|B|-1}{k-1}\geq m|B|.\]

\begin{lcor}\label{main thm 2}
    Let $G=A\times B$ as above, and $k$ be an integer in the range $[3, |B|]$. Assume that the following conditions are satisfied
    \begin{enumerate}
        \item $\binom{|B|-1}{k-1}\geq m|B|$,
        \item For a subset $\{\mb{a}_{i_1}, \dots, \mb{a}_{i_n}\}$ of $\{\mb{a}_1, \dots, \mb{a}_l\}$,  \[\det\mathbb{D}(f_1, \dots, f_n)\neq 0,\] where $f_j:=\op{Tr}_{\mb{a}_j}$.
    \end{enumerate}
    Then, we have that $N_{n}(X;G)\ll X^{k+t-1}$.
\end{lcor}
The next result is a special case of Corollary \ref{main thm 2}.

\begin{lcor}\label{cor B}
    Let $r\in \Z_{\geq 2}$, $n_1, \dots, n_r$ be integers with $n_i\geq 2$ for all $i$. Set $G:=S_{n_1}\times S_{n_2}\times \dots \times S_{n_r}$ and consider the natural embedding
    \[G\hookrightarrow S_{n_1}\times S_{n_2}\times \dots \times S_{n_{r-1}}\times S_{(n_r!)}\hookrightarrow S_{n_1n_2\dots n_{r-1}(n_r!)},\] where the last factor $S_{n_r}$ operates via the regular representation.
    Assume that for $t=2$,
    \begin{enumerate}
        \item $l=\binom{n_r!-1}{2}\geq \left(\prod_{i=1}^{r-1} n_i\right) n_r!$,
        \item $\det \mathbb{D}\neq 0$ for some subset of vectors of $\{\mb{a}_i\mid i\in [1, l]\}$. 
    \end{enumerate}
    Then, we have that $N_{n}(X;G)\ll X^{4}$.
\end{lcor}

Fix $(n_1, \dots, n_{r-1})$, then, for large enough values of $n_r$, the inequality $\binom{n_r!-1}{2}\geq \left(\prod_{i=1}^{r-1} n_i\right) n_r!$ is satisfied. Then, the above bound $N_n(X;G)\ll X^4$ is significantly better than what one is able to derive from the aforementioned asymptotic upper bounds for $N_n(X)$. We acknowledge that the result is only conditional since it assumes the smoothness condition $\det \mathbb{D}\neq 0$. We expect that this condition to hold for most groups in this family. Furthermore, the smoothness condition is indeed very concrete, as our examples show. In section \ref{s 3}, we illustrate Corollary \ref{main thm 2} for two examples. 

\begin{itemize}
    \item First, we consider $G=S_3\hookrightarrow S_6$ via the regular representation.
    \item Second, we take $G=S_3\times D_8$ and consider the embedding that is the following composite  \[G\hookrightarrow S_3\times S_8\hookrightarrow  S_{24},\] with $m=3$ and $q=8$. The second factor of this embedding is via the regular representation of $D_8$. Then, we find that $B=D_8$ has $8$ elements. Thus, taking $k=4$, we find that $\binom{7}{3}=35> 24$. 
\end{itemize} 
These examples only serve to illustrate our results, which are far more general. We do not claim that they cannot be derived from known results. It is perhaps possible to effectively illustrate more elaborate examples illustrating Corollaries \ref{main thm 2} and \ref{cor B}, though this proves to be cumbersome. The Jacobian matrix for the second example itself is a $24\times 24$ polynomial matrix in $24$ variables. 
\subsection{Outlook} Our methods are motivated by the strategy taken in the above mentioned works of Ellenberg-Venkatesh \cite{ellenberg2006number}, Couveignes \cite{couveignes2020enumerating}, and Lemke-Oliver and Thorne \cite{lemke2022upper}. The methods introduced in this manuscript could potentially motivate future developments in the area of number field counting.

\subsection{Acknowledgment} When the project was started, the second named author's research was supported by the CRM-Simons postdoctoral fellowship.

\section{A conditional upper bound}
\par In this section, we establish a conditional asymptotic upper bound for $N_n(X;G)$. This result is based on a numerical criterion that involves the non-vanishing of the determinant of a Jacobian matrix. In the next section, this criterion is demonstrated through an example. 

\subsection{A general criterion}
\par Let $G$ be a transitive subgroup of $S_n$. We assume without loss of generality that there exists a number field $K$ with $[K:\Q]=n$, such that $\op{Gal}(\widetilde{K}/\Q)\simeq G$. Let $\alpha\in \cO_K$ be a primitive element, i.e., $K=\Q(\alpha)$. For $i\in [1, n]$, set $\alpha_i:=\sigma_i(\alpha)$. Recall that for $j\in [1, r]$, the image of $\sigma_j$ is contained in $K$. Observe therefore that $\alpha_i\in \cO_K$ for all $i\in [1, r]$. For $\mb{a}=(a_1, \dots, a_r)\in \Z_{\geq 0}^r$, we recall from Definition \ref{def of tr}, that
\[\op{Tr}_{\mb{a}}(x_1, \dots, x_n):=\sum_{i=1}^n \left(\prod_{j=1}^r x_{\pi_j(i)}^{a_j}\right).\]

We note that 
\[\op{Tr}_{\mb{a}}(\alpha_1, \dots, \alpha_n)=\sum_{i=1}^n \left(\prod_{j=1}^r \sigma_i \sigma_j(\alpha)^{a_j}\right)=\op{Tr}_{L/K}\left(\prod_{j=1}^r \sigma_j(\alpha)^{a_j}\right)\in \cO_K. \]Set $\lVert \alpha\rVert $ to denote the maximum of $|\sigma(\alpha)|$ as $\sigma$ ranges over all embeddings of $L$ into $\bar{\Q}$. Letting $x_\alpha:=(\alpha_1, \dots, \alpha_r)$, we find that 
\[\lVert \op{Tr}_{\mb{a}}(x_\alpha)\rVert\ll \lVert \alpha\rVert^{H(\mb{a})},\]
where $H(\mb{a}):= \sum_{i \in [1,r]}a_i$.
\begin{lemma}\label{lemma 2.2}
    For $N\geq 1$, consider polynomial functions \[f_1, \dots, f_N:\mathbf{A}^N(\mathbb{C})\rightarrow \mathbf{A}^1(\mathbb{C}).\] Assume that the determinant of $\left(\frac{\partial f_i}{\partial x_j}\right)_{1\leq i,j\leq N}$ is not identically zero. Then, there exists a non-zero polynomial $P(x_1, \dots, x_N)$ such that whenever $P(\mathbf{x}_0)\neq 0$, the variety 
    \[V_{\mathbf{x}_0}:=\left\{\mathbf{x}\in \mathbf{A}^N(\mathbb{C})\mid f_i(\mathbf{x})=f_i(\mathbf{x}_0)\text{ for all } i\in [1, N]\right\}\] consists of at most $\prod_i \deg f_i$ points. 
\end{lemma}

\begin{proof}
The above result is \cite[Lemma 2.1]{lemke2022upper}.    
\end{proof}

\begin{proposition}\label{main prop}
    Let $G$ be a transitive permutation group contained in $S_n$ and set $H:=G\cap S_{n-1}$. We set $N$ to be the normalizer of $H$, $r:=|N/H|$ and for $j\in [1, r]$, let $\pi_j\in S_n$ be the associated permutation (cf. Definition \ref{def of pi}). Let $\mb{a}_1, \dots, \mb{a}_n \in \mathbb{Z}_{\geq 0}^r$ be a set of vectors and set $f_i:=\op{Tr}_{\mb{a}_i}$ for $i\in [1, n]$. Assume that the determinant of the Jacobian-matrix \[\mathbb{D}=\mathbb{D}(f_1, \dots, f_n):=\left(\frac{\partial f_i}{\partial x_j}\right)_{1\leq i, j\leq n}\] is not identically $0$. Then we have the asymptotic bound
    \[N_{n}(X;G)\ll X^{\frac{1}{n}\left(\sum_{i=1}^n H(\mb{a}_i)\right)},\] where the implied constant depends only on the vectors $\{\mb{a}_i\mid i\in [1, n]\}$.
\end{proposition}
\begin{proof}
    For $i\in [1, n]$, set $f_i:=\op{Tr}_{\mb{a}_i}$ and $z:=\prod_i \op{deg}f_i$. Since the determinant of $\mathbb{D}$ is not identically zero, it follows from Lemma \ref{lemma 2.2} that there exists a non-zero polynomial $P(x_1, \dots, x_n)$ such that whenever $P(\mathbf{x}_0)\neq 0$, the variety 
    \[V_{\mathbf{x}_0}:=\left\{\mathbf{x}\in \mathbf{A}^n(\mathbb{C})\mid f_i(\mathbf{x})=f_i(\mathbf{x}_0)\text{ for all } i\in [1, n]\right\}\] consists of at most $z$ points. Let $K$ be a number field with $[K:\Q]=n$ and $|\Delta_K|\leq X$ and assume that $\op{Gal}(\widetilde{K}/\Q)\simeq G$ as permutation subgroups of $S_n$. Recall that $r=|\op{Gal}(K/K_0)|$ and $\sigma_1, \dots, \sigma_r$ are embeddings with image in $K$. Let $x_\alpha$ denote the point $(\sigma_1(\alpha), \sigma_2(\alpha), \dots, \sigma_r(\alpha))\in \cO_K^r$. Then it follows from a standard argument (cf. the proof of \cite[Theorem 1.2]{lemke2022upper}) that $\alpha\in \cO_K$ can be chosen such that 
    \begin{enumerate}
        \item $P(x_\alpha)\neq 0$,
        \item $K=\Q(\alpha)$,
        \item $\lVert \alpha\rVert \ll X^{\frac{1}{n}}$.
    \end{enumerate}
    Then there are at most $z$ values $x\in \mathbb{C}^r$ such that $\op{Tr}_{\mb{a}_i}(x)=\op{Tr}_{\mb{a}_i}(x_\alpha)$ for all $i=1, \dots, n$. In particular, the number field $K$ is determined up to $z$ choices by the vector
    \[\left(\op{Tr}_{\mb{a}_1}(x_\alpha), \op{Tr}_{\mb{a}_2}(x_\alpha), \dots, \op{Tr}_{\mb{a}_i}(x_\alpha), \dots, \op{Tr}_{\mb{a}_n}(x_\alpha)\right)\in \Z^n.\] On the other hand, $|\op{Tr}_{\mb{a}_i}(x_\alpha)|\ll X^{\frac{H(\mb{a}_i)}{n}}$ for all $i\in [1, n]$, and thus, the total number of such vectors is at most $\left(\prod_i (2X)^{\frac{H(\mb{a}_i)}{n}}\right)$. We deduce that $N_n(X;G)\ll \prod_i X^{\frac{H(\mb{a}_i)}{n}}$, where the implied constant depends only on the vectors $\mb{a}_i$. 
\end{proof}
With respect to notation from Proposition \ref{main prop}, assume that $r\geq 3$. Fix integers $k$ and $t$ such that $k\in [2, r)$ and $t\geq 2$. Let $\mathcal{S}_k(r)$ be the  all $k$-element subsets of $[r]$ that contain $1$. Given $\mathcal{B}\in \mathcal{S}_k(r)$, let $\mb{a}(\mathcal{B})=(a_1, \dots, a_r)$ be the vector which is defined so that 
\[a_i:=\begin{cases}
    t\text{ if }i=1;\\
    1\text{ if }i\neq 1\text{ and }i \in \mathcal{B};\\
    0\text{ otherwise. }
\end{cases}\]
\begin{definition}\label{def of ai}Set $l:=\# \mathcal{S}_k(r)=\binom{r-1}{k-1}$ and assume that $l\geq n$. Write $\mathcal{B}_1, \dots, \mathcal{B}_l$ to be an enumeration of the subsets in $\mathcal{S}_k(r)$ and set $\mb{a}_i:=\mb{a}(\mathcal{B}_i)$. 
\end{definition}

\begin{proposition}
    With respect to notation above, the trace functions $f_i:=\op{Tr}_{\mb{a}_i}$ are linearly independent over $\mathbb{C}$.
\end{proposition}
\begin{proof}
    We write $\mb{a}_1=(b_1, \dots, b_r)$, with $b_1=2$ and $b_j\in \{0,1\}$ for $j\in [2, r]$. Note that the monomial $g:=x_1^{b_1}x_2^{b_2}\dots x_r^{b_r}$ is the support of $f_1$. It suffices to show that this monomial is not in the support of any of the polynomials $f_i$ for $i\in [2, l]$. Write $\mb{a}_i:=(c_1, \dots, c_r)$, once again with $c_1=2$ and $c_i\in \{0,1\}$ for $i\in [2, r]$. Then, any monomial in the support of $f_i$ is of the form $h:=x_{\pi_1(j)}^{c_1}x_{\pi_2(j)}^{c_2}\dots x_{\pi_r(j)}^{c_r}$ for some $j\in [1, n]$. Note that $\pi_1(j)=j$, and $c_1=2$. Therefore, in orderfor $g=h$, it must be the case that $j=1$. This implies that $\mb{a}_i=\mb{a}_1$, which is a contradiction. Therefore, none of the monomials in the support of $f_i$ coincide with $g$. This implies that the functions $f_1, \dots, f_l$ are linearly independent over $\mathbb{C}$.
\end{proof}
We end the section with the proofs of Theorem \ref{main thm} and Corollary \ref{main thm 2}.

\begin{proof}[Proof of Theorem \ref{main thm}]
    We find that $H(\mb{a}_i)=k+t-1$ for all $i\in [1, l]$, and thus, the result is a direct consequence of Proposition \ref{main prop}.
\end{proof}

\begin{proof}[Proof of Corollary \ref{main thm 2}]
    The result follows directly from Theorem \ref{main thm}.
\end{proof}

\section{Computations verifying the Jacobian condition}\label{s 3}

\subsection{Example 1: $S_3\subset S_6$}
\par We take $G=S_3$ sitting inside $S_6$ via the regular representation. The permutations of $3$ elements are $g_1=\mathbf{1}$, $g_2=(12)$, $g_3=(23)$, $g_4=(13)$, $g_5=(123)$ and $g_6=(132)$. In this case, $H=1$ and $N=G$, we find that $r=6$. Given an $S_3$-extension, $K/\Q$, we identify $g_i$ with an embedding $\sigma_i:K\hookrightarrow \bar{\Q}$. From the multiplication table for $S_3$, one is able to determine that 
\[\begin{split}
   & \pi_1=\begin{pmatrix}
    1 & 2 & 3 & 4 & 5 & 6 \\
    1 & 2 & 3 & 4 & 5 & 6
  \end{pmatrix}, 
   \pi_2=\begin{pmatrix}
    1 & 2 & 3 & 4 & 5 & 6 \\
    2 & 1 & 6 & 5 & 4 & 3
  \end{pmatrix}, 
   \pi_3=\begin{pmatrix}
    1 & 2 & 3 & 4 & 5 & 6 \\
    3 & 5 & 1 & 6 & 2 & 4
  \end{pmatrix}, \\
  &     \pi_4=\begin{pmatrix}
    1 & 2 & 3 & 4 & 5 & 6 \\
    4 & 6 & 5 & 1 & 3 & 2
  \end{pmatrix}, 
   \pi_5=\begin{pmatrix}
    1 & 2 & 3 & 4 & 5 & 6 \\
    5 & 3 & 4 & 2 & 6 & 1
  \end{pmatrix}, 
   \pi_6=\begin{pmatrix}
    1 & 2 & 3 & 4 & 5 & 6 \\
    6 & 4 & 2 & 3 & 1 & 5
  \end{pmatrix}.
\end{split}\]
We take $t=2$, $k=3$ and note that $l=\binom{r-1}{k-1}=\binom{5}{2}=10>6$. Let us list the set $\{\mb{a}_i\mid i\in [1, 10]\}$. We find that 
\[\begin{split}
    & \mb{a}_1=(2,1, 1, 0, 0, 0), \mb{a}_2=(2,1, 0, 1, 0, 0), \mb{a}_3=(2,1, 0, 0, 1, 0), \\
    & \mb{a}_4=(2,1, 0, 0, 0, 1),\mb{a}_5=(2,0, 1, 1, 0, 0),\mb{a}_6=(2,0, 1, 0, 1, 0),\\
    & \mb{a}_7=(2,0, 1, 0, 0, 1),\mb{a}_8=(2,0, 0, 1, 1, 0),\mb{a}_9=(2,0, 0, 1, 0, 1),\\
    & \mb{a}_{10}=(2,0,0,0,1,1).
\end{split}\]
Let $\mb{a}=(a_1, \dots, a_6)$, we have that 
\[\begin{split}\op{Tr}_{\mb{a}}=& x_1^{a_1}x_2^{a_2}x_3^{a_3}x_4^{a_4}x_5^{a_5}x_6^{a_6}+x_2^{a_1}x_1^{a_2}x_5^{a_3}x_6^{a_4}x_3^{a_5}x_4^{a_6}+x_3^{a_1}x_6^{a_2}x_1^{a_3}x_5^{a_4}x_4^{a_5}x_2^{a_6} \\
+ & x_4^{a_1}x_5^{a_2}x_6^{a_3}x_1^{a_4}x_2^{a_5}x_3^{a_6}+x_5^{a_1}x_4^{a_2}x_2^{a_3}x_3^{a_4}x_6^{a_5}x_1^{a_6}+x_6^{a_1}x_3^{a_2}x_4^{a_3}x_2^{a_4}x_1^{a_5}x_5^{a_6}.\end{split}\]

Setting $f_i:=\op{Tr}_{\mb{a}_i}$, we compute the jacobian of $(f_1, \dots, f_6)$ in the variables $(x_1, \dots, x_6)$, and find that its determinant is not identically $0$. This computation was performed on the {\tt SageMathCloud}, the code is provided below
\begin{verbatim}
var('x1,x2,x3,x4,x5,x6')
f1 = x1^2*x2*x3+x2^2*x1*x6+x3^2*x5*x1 +x4^2*x6*x5+x5^2*x3*x4+x6^2*x4*x2
f2 = x1^2*x2*x4+x2^2*x1*x6+x3^2*x6*x5 +x4^2*x5*x1+x5^2*x4*x3+x6^2*x3*x4
f3 = x1^2*x2*x5+x2^2*x1*x3+x3^2*x6*x4 +x4^2*x5*x2+x5^2*x4*x6+x6^2*x3*x1
f4 = x1^2*x2*x6+x2^2*x1*x4+x3^2*x6*x2 +x4^2*x5*x3+x5^2*x4*x1+x6^2*x3*x5
f5=x1^2*x3*x4+x2^2*x5*x6+x3^2*x1*x5 +x4^2*x6*x1+x5^2*x2*x3+x6^2*x4*x3
f6=x1^2*x3*x5+x2^2*x5*x3+x3^2*x1*x4 +x4^2*x6*x2+x5^2*x2*x6+x6^2*x4*x1
b=jacobian( [f1,f2,f3,f4, f5,f6], [x1,x2,x3,x4,x5,x6])
a=det(b)
print(a)
\end{verbatim}
The conditions of Theorem \ref{main thm} are satisfied in this case.

\subsection{Example 2:} We set $G:=S_3\times D_8\subset S_{24}$, where $D_8\subset S_8$ via the regular representation. We note that $D_8$ is a subgroup of $S_4$. It consists of the rotations of the square $1, a, a^2, a^3$, and the reflections $b,ab,a^2b,a^3b$. We order these elements $g_1, \dots, g_4$ and $g_5, \dots, g_8$ respectively.
\par We order the set $\{(i,j)\mid i\in [1, 3], j\in [1, 8]\}$ in lexicographic order, with $D_4$ acting on the second components. Thus, $x_1,x_2, \dots, x_{24}=x_{(1, 1)}, \dots, x_{(3, 8)}$. For ease of notation, we set 
\[\begin{split}
    & u_1, \dots, u_8=x_1, \dots, x_8;\\
    & v_1, \dots, v_8= x_9, \dots, x_{16};\\
    & w_1, \dots, w_8=x_{17}, \dots, x_{24}.
\end{split}\]
By inspecting the multiplication tables for $D_8$, we find that 
\[\begin{split}\op{Tr}_{\mb{a}}=& u_1^{a_1}u_2^{a_2}u_3^{a_3}u_4^{a_4}u_5^{a_5}u_6^{a_6}u_7^{a_7}u_8^{a_8}+u_2^{a_1}u_3^{a_2}u_4^{a_3}u_1^{a_4}u_6^{a_5}u_7^{a_6}u_8^{a_7}u_5^{a_8} \\ 
+& u_3^{a_1}u_4^{a_2}u_1^{a_3}u_2^{a_4}u_7^{a_5}u_8^{a_6}u_5^{a_7}u_6^{a_8}+u_4^{a_1}u_1^{a_2}u_2^{a_3}u_3^{a_4}u_8^{a_5}u_5^{a_6}u_6^{a_7}u_7^{a_8}\\
+& u_5^{a_1}u_8^{a_2}u_7^{a_3}u_6^{a_4}u_1^{a_5}u_4^{a_6}u_3^{a_7}u_2^{a_8}+u_6^{a_1}u_5^{a_2}u_8^{a_3}u_7^{a_4}u_2^{a_5}u_1^{a_6}u_4^{a_7}u_3^{a_8}\\
+ & u_7^{a_1}u_6^{a_2}u_5^{a_3}u_8^{a_4}u_3^{a_5}u_2^{a_6}u_1^{a_7}u_4^{a_8}+u_8^{a_1}u_7^{a_2}u_6^{a_3}u_5^{a_4}u_4^{a_5}u_3^{a_6}u_2^{a_7}u_1^{a_8} \\
+ & v_1^{a_1}v_2^{a_2}v_3^{a_3}v_4^{a_4}v_5^{a_5}v_6^{a_6}v_7^{a_7}v_8^{a_8}+v_2^{a_1}v_3^{a_2}v_4^{a_3}v_1^{a_4}v_6^{a_5}v_7^{a_6}v_8^{a_7}v_5^{a_8} \\ 
+& v_3^{a_1}v_4^{a_2}v_1^{a_3}v_2^{a_4}v_7^{a_5}v_8^{a_6}v_5^{a_7}v_6^{a_8}+v_4^{a_1}v_1^{a_2}v_2^{a_3}v_3^{a_4}v_8^{a_5}v_5^{a_6}v_6^{a_7}v_7^{a_8}\\
+& v_5^{a_1}v_8^{a_2}v_7^{a_3}v_6^{a_4}v_1^{a_5}v_4^{a_6}v_3^{a_7}v_2^{a_8}+v_6^{a_1}v_5^{a_2}v_8^{a_3}v_7^{a_4}v_2^{a_5}v_1^{a_6}v_4^{a_7}v_3^{a_8}\\
+ & v_7^{a_1}v_6^{a_2}v_5^{a_3}v_8^{a_4}v_3^{a_5}v_2^{a_6}v_1^{a_7}v_4^{a_8}+v_8^{a_1}v_7^{a_2}v_6^{a_3}v_5^{a_4}v_4^{a_5}v_3^{a_6}v_2^{a_7}v_1^{a_8} \\ 
+& w_1^{a_1}w_2^{a_2}w_3^{a_3}w_4^{a_4}w_5^{a_5}w_6^{a_6}w_7^{a_7}w_8^{a_8}+w_2^{a_1}w_3^{a_2}w_4^{a_3}w_1^{a_4}w_6^{a_5}w_7^{a_6}w_8^{a_7}w_5^{a_8} \\ 
+& w_3^{a_1}w_4^{a_2}w_1^{a_3}w_2^{a_4}w_7^{a_5}w_8^{a_6}w_5^{a_7}w_6^{a_8}+w_4^{a_1}w_1^{a_2}w_2^{a_3}w_3^{a_4}w_8^{a_5}w_5^{a_6}w_6^{a_7}w_7^{a_8}\\
+& w_5^{a_1}w_8^{a_2}w_7^{a_3}w_6^{a_4}w_1^{a_5}w_4^{a_6}w_3^{a_7}w_2^{a_8}+w_6^{a_1}w_5^{a_2}w_8^{a_3}w_7^{a_4}w_2^{a_5}w_1^{a_6}w_4^{a_7}w_3^{a_8}\\
+ & w_7^{a_1}w_6^{a_2}w_5^{a_3}w_8^{a_4}w_3^{a_5}w_2^{a_6}w_1^{a_7}w_4^{a_8}+w_8^{a_1}w_7^{a_2}w_6^{a_3}w_5^{a_4}w_4^{a_5}w_3^{a_6}w_2^{a_7}w_1^{a_8}.\end{split}\]

We take $k=4$ and thus find that $l=\binom{7}{3}=35$. We choose a set of $24$ vectors $\mb{a}_1, \dots, \mb{a}_{24}$ and list them below
\[\begin{split}
    & \mb{a}_1=(2,1,1,1,0,0,0,0), \mb{a}_2=(2,1,1,0,1,0,0,0), \mb{a}_3=(2,1,1,0,0,1,0,0),\\
    & \mb{a}_4=(2,1,1,0,0,0,1,0), \mb{a}_5=(2,1,1,0,0,0,0,1), \mb{a}_6=(2,1,0,1,1,0,0,0),\\
     & \mb{a}_7=(2,1,0,1,0,1,0,0), \mb{a}_8=(2,1,0,1,0,0,1,0), \mb{a}_9=(2,1,0,1,0,0,0,1),\\
      & \mb{a}_{10}=(2,1,0,0,1,1,0,0), \mb{a}_{11}=(2,1,0,0,1,0,1,0), \mb{a}_{12}=(2,1,0,0,1,0,0,1),\\
      & \mb{a}_{13}=(2,1,0,0,0,1,1,0), \mb{a}_{14}=(2,1,0,0,0,1,0,1), \mb{a}_{15}=(2,1,0,0,0,0,1,1),\\
      & \mb{a}_{16}=(2,0,1,1,1,0,0,0), \mb{a}_{17}=(2,0,1,1,0,1,0,0), \mb{a}_{18}=(2,0,1,1,0,0,1,0),\\
       & \mb{a}_{19}=(2,0,1,1,0,0,0,1), \mb{a}_{20}=(2,0,1,0,1,1,0,0), \mb{a}_{21}=(2,0,1,0,1,0,1,0),\\
       & \mb{a}_{22}=(2,0,1,0,1,0,0,1), \mb{a}_{23}=(2,0,1,0,0,1,1,0), \mb{a}_{24}=(2,0,1,0,0,1,0,1).
\end{split}\]
Setting $f_i:=\op{Tr}_{\mb{a}_i}$, consider the Jacobian matrix of $(f_1, \dots, f_{24})$ in the variables \[(u_1, \dots, u_8, v_1, \dots, v_8, w_1, \dots, w_8).\] If the associated Jabobian matrix is shown to be nonsingular, then the Corollary \ref{main thm 2} implies that $N_{24}(X;S_3\times D_8)\ll X^5$. Compare this with Schmidt's upper bound, which implies that $N_{24}(X; S_3\times D_8)\ll X^{6.5}$.
\subsection*{Data availability statement} No data was generated or analyzed in establishing our results.
\bibliographystyle{alpha}
\bibliography{references}
\end{document}